\documentclass[preprint]{elsarticle}
\usepackage{amsmath,amsthm,verbatim,enumitem,xcolor,longtable,url} 
\usepackage{tikz}
\usepackage{tkz-berge,tkz-graph}

\setlist{itemsep=0ex}
\usetikzlibrary{matrix,arrows,decorations.pathreplacing,decorations.pathmorphing,positioning,patterns}
\tikzset{baseline=(current bounding box.west)}
\tikzset{every matrix/.style={execute at begin cell=\node\bgroup$,execute at end cell=$\egroup;,minimum size=4mm,matrix anchor=south west,inner sep=0pt}}

\definecolor{Light}{RGB}{190,190,190}  
\definecolor{Dark}{RGB}{198,171,27}

\usepackage[T1]{fontenc}
\usepackage[bitstream-charter]{mathdesign}
\usepackage{todonotes} 

\newtheorem{theorem}{Theorem}
 
\newtheorem{corollary}[theorem]{Corollary}
\newtheorem{conjecture}[theorem]{Conjecture}
\theoremstyle{definition}
\newtheorem{remark}[theorem]{Remark}

\newtheorem{example}[theorem]{Example}

\begin{document}

\begin{frontmatter}
\title{Upper Bounds on Sets of Orthogonal Colorings of Graphs}
\author{Serge C. Ballif}
\address{Nevada State College\\ \url{serge.ballif@nsc.edu}}

\begin{abstract}
We generalize the notion of orthogonal latin squares to colorings of simple graphs. Two $n$-colorings of a graph are said to be \emph{orthogonal} if whenever two vertices share a color in one coloring they have distinct colors in the other coloring. We show that the usual bounds on the maximum size of a certain set of orthogonal latin structures such as latin squares, row latin squares, equi-$n$ squares, single diagonal latin squares, double diagonal latin squares, or sudoku squares are a special cases of bounds on orthogonal colorings of graphs. 
\end{abstract}  

\begin{keyword}
orthogonal latin squares \sep graph coloring \sep chromatic number
\MSC[2010] 05B15 \sep 05C15
\end{keyword}
\end{frontmatter}

%\linenumbers

\section{Introduction}
A \emph{latin square of order $n$} is an $n\times n$ array filled with $n$ symbols such that each symbol occurs exactly once in each row and column. Two latin squares of order $n$ are said to be \emph{orthogonal} if, when superimposed, each of the $n^2$ possible ordered pairs of symbols occur. A set of pairwise orthogonal latin squares is said to be a set of \emph{mutually orthogonal latin squares} (MOLS).

We use the notation $N(n)$ for the largest size of a set of MOLS of order $n$. It is well known that $N(n)\le n-1$. When $n$ is a power of a prime one can construct a set of $n-1$ MOLS using finite fields, so when $n$ is a prime power $N(n)=n-1$ (see \cite[Ch. 2]{LM98}). In this paper we introduce a generalization of the function $N(n)$ that gives an upper bound on the cardinality of a set of pairwise orthogonal $n$-colorings of a graph.

\begin{example} \label{Polyomino Example}
A \emph{polyomino}\label{D: polyomino} is a (possibly disconnected) subset of cells from a square tiling of the plane. If we label the cells of a polyomino so that no symbol occurs more than once in any row or column, then we get a \emph{latin polyomino}.\label{D: latin polyomino} As an example, below we exhibit five latin polyominoes in two polyomino shapes.
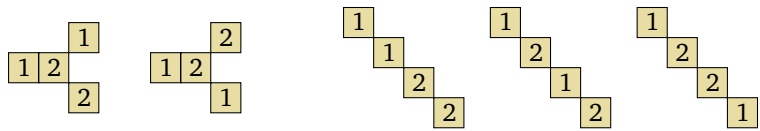
\[
%\includegraphics{UpperBoundsOnSetsOfOrthogonalColoringsOfGraphs-figure0.pdf}
%\qquad
%\includegraphics{UpperBoundsOnSetsOfOrthogonalColoringsOfGraphs-figure1.pdf}
%\qquad\qquad
%\includegraphics{UpperBoundsOnSetsOfOrthogonalColoringsOfGraphs-figure2.pdf}
%\qquad
%\includegraphics{UpperBoundsOnSetsOfOrthogonalColoringsOfGraphs-figure3.pdf}
%\qquad
%\includegraphics{UpperBoundsOnSetsOfOrthogonalColoringsOfGraphs-figure4.pdf}
%\begin{comment}
\begin{tikzpicture}[scale=.4]
\begin{scope}
\draw[fill=Dark!40] (0,1) rectangle (1,2) rectangle (2,1) rectangle (3,0);
\draw[fill=Dark!40] (2,2) rectangle (3,3);
\matrix 
{
&&1\\
1&2&\\
&&2\\
};
\end{scope}
\end{tikzpicture}
\qquad
\begin{tikzpicture}[scale=.4]
\begin{scope}
\draw[fill=Dark!40] (0,1) rectangle (1,2) rectangle (2,1) rectangle (3,0);
\draw[fill=Dark!40] (2,2) rectangle (3,3);
\matrix
{
&&2\\
1&2&\\
&&1\\
};
\end{scope}
\end{tikzpicture}
%=================
\qquad\qquad
\begin{tikzpicture}[scale=.4]
\begin{scope}
\draw[fill=Dark!40] (0,4) rectangle (1,3) rectangle (2,2) rectangle (3,1) rectangle (4,0);
\matrix
{
1&&&\\
&1&&\\
&&2&\\
&&&2\\
};
\end{scope}
\end{tikzpicture}
\quad
\begin{tikzpicture}[scale=.4]
\begin{scope}
\draw[fill=Dark!40] (0,4) rectangle (1,3) rectangle (2,2) rectangle (3,1) rectangle (4,0);
\matrix
{
1&&&\\
&2&&\\
&&1&\\
&&&2\\
};
\end{scope}
\end{tikzpicture}
\quad
\begin{tikzpicture}[scale=.4]
\begin{scope}
\draw[fill=Dark!40] (0,4) rectangle (1,3) rectangle (2,2) rectangle (3,1) rectangle (4,0);
\matrix
{
1&&&\\
&2&&\\
&&2&\\
&&&1\\
};
\end{scope}
\end{tikzpicture}
%\end{comment}
\]
Note that the two latin polyominoes on the left are orthogonal in the sense that when they are superimposed each of the ordered pairs is distinct. Similarly, the three latin polyominoes on the right are mutually orthogonal. In each case we have exceeded the bound $N(2)=1$ that occurs in the case where the polyomino shape is a square. 
\end{example}

In this paper we consider latin squares as a special case of a proper $n$-coloring of a simple graph. We view orthogonality of latin squares as a special case of orthogonal colorings of a graph, a notion first studied in \cite{CY99}. As a corollary to our main theorems we obtain bounds on the cardinality of orthogonal sets of several structures including equi-$n$ squares, row or column latin squares, latin squares, single diagonal latin squares, double diagonal latin squares, and partial latin squares.

A \emph{graph}\label{D: graph}, $G$, is a set of \emph{vertices}\label{D: vertex}, $V(G)$, and a set of \emph{edges}, $E(G)$, where each edge is an unordered pair of distinct elements of $V(G)$. The \emph{order} of $G$ is the number of vertices in $V(G)$. Two vertices are said to be \emph{adjacent}\label{D: adjacent} (or \emph{neighbors}) if they are connected by an edge in $E$. We write $u\sim v$ to signify that $u$ is adjacent to $v$. The \emph{degree}\label{D: degree} of a vertex is the number of vertices adjacent to it. An \emph{$r$-clique} is a (sub)graph of $r$ pairwise adjacent vertices. % The \emph{complement}\label{D: complement} of a graph $G$ is a graph $\overline{G}$\label{complement of G} that has vertex set $V(G)$, and two vertices of $\overline{G}$ are adjacent if and only if they are not adjacent in $G$.

A \emph{(proper) coloring} of $G$ is a labeling of the vertices so that any two adjacent vertices have distinct labels (which we call \emph{colors}). An \emph{$n$-coloring}\label{D: n-coloring} is a coloring consisting of (at most) $n$ colors. It is usually convenient to assume these colors are the numbers $1,\ldots,n$.  The \emph{chromatic number}\label{D: chromatic number} $\chi=\chi(G)$, of $G$ is the minimum number, $n$, such that there exists an $n$-coloring of $G$. If $C$ is a coloring of $G$, then we write $C(v)$ for the color that $C$ assigns to $v$.

A latin square is an $n$-coloring of an $n\times n$ rook's graph. An \emph{$m\times n$ rook's graph} is a graph with $mn$ vertices in $m$ rows and $n$ columns, where two vertices are adjacent if they share a row or a column.

Two colorings, $C_1$ and $C_2$, of a graph, $G$, are said to be \emph{orthogonal}\label{D: orthogonal colorings} if whenever two distinct vertices share a color in $C_1$ they have distinct colors in $C_2$. A set of pairwise orthogonal $n$-colorings of $G$ is called a set of \emph{mutually orthogonal $n$-colorings}. We denote the maximum size of a set of mutually orthogonal $n$-colorings of the graph $G$ by \emph{$N(G,n)$}.\label{N(G,n)} 

\begin{example}\label{Orthogonal Colorings Example}
Below we display orthogonal colorings of a graph of order 9 and a graph of order 16.
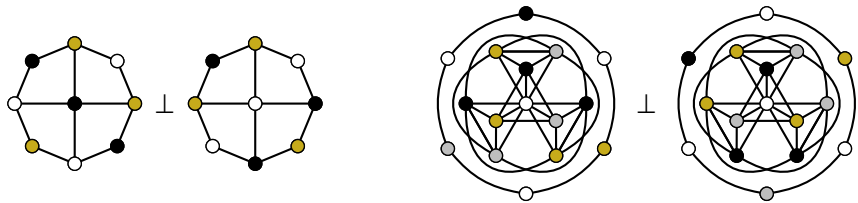
\[
%\includegraphics{UpperBoundsOnSetsOfOrthogonalColoringsOfGraphs-figure5.pdf}
%\begin{comment}
\begin{tikzpicture}[scale=.40]
%\GraphInit[vstyle=Welsh]
\SetVertexSimple[MinSize = 5pt,
LineWidth = .5pt]
\tikzset{VertexStyle/.append style =
{inner sep = 0pt,%
outer sep = 0pt}}
\SetVertexNoLabel
\begin{scope}
\begin{scope}[rotate=30]
\Vertices[unit=3]{circle}{02,30,03,10,01,20}
\end{scope}
\begin{scope}[rotate=60]
\Vertices[unit=2]{circle}{11,22,32,13,23,31}
\end{scope}
\begin{scope}[rotate=90]
\Vertices[unit=1.15]{circle}{33,21,12} 
\end{scope}
\coordinate (00) at (0,0);
\Vertices[Node]{line}{00}
\Edges(00,11,22,33,00,22,11,33)
\Edges(00,31,23,12,00,23,31,12)
\Edges(00,13,32,21,00,32,13,21)
\Edges(12,21,33,12)
\Edges[style={bend left=60,looseness=1.5}](31,13,22,31)
\Edges[style={bend left=60,looseness=1.5}](23,32,11,23)
\Edges[style={bend left=20}](30,02,20,01,10,03,30)
\AddVertexColor{white}{00,01,02,03}
\AddVertexColor{Light}{10,11,12,13}
\AddVertexColor{Dark}{20,21,22,23}
\AddVertexColor{black}{30,31,32,33}
\draw (4,0) node{$\perp$};
\end{scope}
\begin{scope}[xshift=8cm]
\begin{scope}[rotate=30]
\Vertices[unit=3]{circle}{02,30,03,10,01,20}
\end{scope}
\begin{scope}[rotate=60]
\Vertices[unit=2]{circle}{11,22,32,13,23,31}
\end{scope}
\begin{scope}[rotate=90]
\Vertices[unit=1.15]{circle}{33,21,12} 
\end{scope}
\coordinate (00) at (0,0);
\Vertices[Node]{line}{00}
\Edges(00,11,22,33,00,22,11,33)
\Edges(00,31,23,12,00,23,31,12)
\Edges(00,13,32,21,00,32,13,21)
\Edges(12,21,33,12)
\Edges[style={bend left=60,looseness=1.5}](31,13,22,31)
\Edges[style={bend left=60,looseness=1.5}](23,32,11,23)
\Edges[style={bend left=20}](30,02,20,01,10,03,30)
\AddVertexColor{white}{00,10,20,30}
\AddVertexColor{Light}{01,11,21,31}
\AddVertexColor{Dark}{02,12,22,32}
\AddVertexColor{black}{03,13,23,33}
\end{scope}
%============
\begin{scope}[xshift=-15cm]
\begin{scope}[rotate=45]
\Vertices[unit=2]{circle}{aa,bb,cc,ab,ba,ac,cb,bc}
\end{scope}
\coordinate (ca) at (0,0);
\Vertices[Node]{line}{ca}
\Edges(aa,bb,cc,ab,ba,ac,cb,bc,aa)
\Edges(bc,ca,ab)
\Edges(bb,ca,ac)
\AddVertexColor{white}{aa,ab,ac}
\AddVertexColor{Dark}{ba,bb,bc}
\AddVertexColor{black}{ca,cb,cc}
\draw (3,0) node{$\perp$};
\begin{scope}[xshift=6cm]
\begin{scope}[rotate=45]
\Vertices[unit=2]{circle}{aa,bb,cc,ab,ba,ac,cb,bc}
\end{scope}
\coordinate (ca) at (0,0);
\Vertices[Node]{line}{ca}
\Edges(aa,bb,cc,ab,ba,ac,cb,bc,aa)
\Edges(bc,ca,ab)
\Edges(bb,ca,ac)
\AddVertexColor{white}{aa,ba,ca}
\AddVertexColor{Dark}{ab,bb,cb}
\AddVertexColor{black}{ac,bc,cc}
\end{scope}
\end{scope}
\end{tikzpicture}
%\end{comment}
\]
\end{example}

A graph with a pair of orthogonal $n$-colorings can have at most $n^2$ vertices because there are at most $n^2$ distinct ordered pairs of labels. In fact, if any color shows up more than $n$ times in a coloring, then the coloring has no orthogonal mate.  At the other extreme we may have an overabundance of colors---if $n\ge|V(G)|$ then each vertex can be assigned a different color. Such a coloring is orthogonal to any coloring, including itself.

In the reference, \cite{CY99}, Caro and Yuster investigate the number of colors needed for a graph to have orthogonal colorings. The authors define the \emph{orthogonal chromatic number of $G$}, $O\!\chi(G)$,\label{Ochi(G)} to be the minimum number of colors in any pair of orthogonal $O\!\chi(G)$-colorings of $G$.  Similarly, the \emph{$k$-orthogonal chromatic number of $G$}, $O\!\chi_k(G)$, is the minimum number of colors required so that there exist $k$ mutually orthogonal $O\!\chi_k(G)$-colorings of $G$. The authors find upper bounds on the values of $O\!\chi(G)$ and $O\!\chi_k(G)$ in terms of the parameters of $G$ such as the maximum degree of any vertex of $G$ or the chromatic number of $G$. They also show that several classes of graphs always have the lowest possible value $O\!\chi(G)=\left\lceil \sqrt{|V(G)|}\right\rceil$. 

Here we are concerned with the function $N(G,n)$, which is related to the function $O\!\chi_k(G)$ by the inequalities 
\[
N(G,O\!\chi_k(G))\ge k\qquad\text{and}\qquad O\!\chi_{N(G,n)}(G)\le n.
\]
By studying the function $N(G,n)$ we obtain results that are complementary to the bounds of \cite{CY99} in that they yield lower bounds for $O\!\chi_k(G)$ . In Section~\ref{Orthogonal Chromatic Number Section} we state a few of these lower bounds for $O\!\chi_k(G)$ as corollaries to our main theorems.

Many authors have studied edge colorings of graphs as opposed to vertex colorings. A \emph{proper edge coloring} of a graph is a labeling of the edges of a graph such that any two edges that share a vertex receive distinct labels. For each graph, $G$, the \emph{line graph of $G$} is the graph with vertex set $E(G)$, and two vertices are adjacent if and only if they share a vertex in $G$. Each edge coloring of $G$ corresponds to a vertex coloring of the line graph of $G$ in a natural way, so any question about edge coloring can be converted into a question about vertex coloring. The converse statement does not hold because most graphs are not the line graph of any graph. 

Despite the fact that each edge coloring may be viewed as an example of vertex coloring, many questions about edge colorings are interesting in their own right and have been studied in the context of edge colorings. Of particular note to us is the article \cite{ADH85} by Archdeacon and Dinitz where they studied orthogonal edge colorings of graphs. In their article they viewed a latin square of order $n$ as an edge coloring of the complete bipartite graph $K_{n,n}$. Results about orthogonal edge colorings can be converted to results about orthogonal vertex coloring by studying the corresponding line graph. 

\begin{remark}
A \emph{subgraph}\label{D: subgraph}, $H$, of a graph, $G$, is a graph such that $V(H)\subseteq V(G)$ and $E(H)\subseteq E(G)$. It will always be the case that if $H$ is a subgraph of $G$, then $N(H,n)\ge N(G,n)$ because mutually orthogonal colorings of $G$ are also mutually orthogonal colorings of $H$.
\end{remark}

Our main results are four theorems that give upper bounds for $N(G,n)$, namely  the \emph{degree bound} of Theorem~\ref{Degree Bound Theorem}, the \emph{clique bound} of Theorem~\ref{Clique Bound Theorem}, the \emph{average-degree bound} of Theorem~\ref{Average-Degree Bound Theorem}, and the \emph{edge bound} of Theorem~\ref{Edge Bound Theorem}. The first two theorems give bounds that depend on a large number of edges locally. 

\begin{theorem}[degree bound]\label{Degree Bound Theorem}
Let $G$ be a graph of order $n^2-m$ where $0\le m< n-1$, and let $\Delta$ be the maximum degree of any vertex of $G$. If $\Delta\ge n^2-n$, then $N(G,n)\le1$. Otherwise
\[
N(G,n)\le \left\lfloor\frac{n^2-m-\Delta-1}{n-m-1}\right\rfloor.
\]
\end{theorem}

\begin{theorem}[clique bound]\label{Clique Bound Theorem}
Let $r$, $s$, and $n$ be integers satisfying $1<r, s\le n<r+s$. Let $G$ be a graph with an $s$-clique, $A$, that is disjoint from an $r$-clique, $B$, such that each vertex in $B$ is adjacent to at least $j$ vertices in $A$. Then
\[
N(G,n)\le \left\lfloor\frac{r(s-j)}{r+s-n}\right\rfloor.
\]
\end{theorem}

The next two theorems give upper bounds on $N(G,n)$ based on the number of edges globally.

\begin{theorem}[average-degree bound]\label{Average-Degree Bound Theorem}
Let $G$ be a graph of order $v$ with average vertex degree $D$. Then for $n<v$,
\[
N(G,n)\le\left\lfloor\frac{(v-D-1)n}{v-n}\right\rfloor.
\]
\end{theorem}

\begin{theorem}[edge bound]\label{Edge Bound Theorem}
Let $G$ be a graph with $v$ vertices and $e$ edges, $v>n$. Write $v=qn+r$ with $0\le r< n$. Then
\[
N(G,n)\le \left\lfloor \frac{\binom{v}{2}-e}{(n-r)\binom{\lfloor \frac{v}{n}\rfloor}{2}+r\binom{\lceil \frac{v}{n}\rceil}{2}}\right\rfloor.
\]
\end{theorem}

If $D$ is the average degree of the vertices of $G$, and $e$ is the number of edges in $G$, then $D=\frac{2e}{|V(G)|}$, so the edge bound and the average-degree bound depend on the same information. In fact the edge bound will always perform as well as the average-degree bound. We have included the average-degree bound because of its simplicity and because it is sufficiently strong to give tight bounds for most latin structures.

In Section~\ref{Section: Bounds for latin structures} we provide several corollaries to Theorems~\ref{Clique Bound Theorem}--\ref{Edge Bound Theorem}. We shall show that the best possible upper bounds on the cardinality of sets of mutually orthogonal latin structures can be derived from the four bounds we have introduced. Table~\ref{short table} on page~\pageref{short table} summarizes many of the results of Section~\ref{Section: Bounds for latin structures} by showing the upper bounds on the cardinality of sets of orthogonal latin structures. 

In Section~\ref{sec low bounds} we use an upper bound on $O\!\chi_k(G)$ from \cite{CY99} to obtain a lower bound for $N(G,n)$. In Section~\ref{Orthogonal Chromatic Number Section} we obtain lower bounds on the $k$-orthogonal chromatic number $O\!\chi_k(G)$. In Section~\ref{Section: Proofs} we prove each of the four bounds. In Section~\ref{Section: Constructions} we introduce some graph constructions that arise from sets of mutually orthogonal colorings of a graph.

\noindent \kern-1cm
\newcommand{\columnbox}[1]{\begin{minipage}{1in}\vspace{.05cm}  \begin{center}#1\end{center}\end{minipage}\vspace{.05cm}}
\newcommand{\rowbox}[1]{\begin{minipage}{2cm}\begin{center}#1\end{center}\end{minipage}} 
\begin{table}[h]
\caption{We list the upper bounds for $N(G,n)$ (the maximum cardinality of mutually orthogonal sets of latin structures). Upper bounds that are known to be best possible are in bold.}   \label{short table}
\begin{tabular}{ccccc}
 &\kern-.5cm \rowbox{Degree Bound}%\hspace{.3cm} 
 	&\kern-.1cm \rowbox{Average-Degree Bound}\hspace{.3cm} 
 		&\kern-.5cm \rowbox{Clique Bound} 
 			&\kern-.5cm \rowbox{Edge Bound}\\
Graph $G$ 
	&\kern-.5cm $\left\lfloor\frac{n^2-m-\Delta-1}{n-m-1}\right\rfloor$ 
		&\kern-.1cm $\left\lfloor\frac{(v-D-1)n}{v-n}\right\rfloor$
			&\kern-.5cm $\left\lfloor\frac{r(s-j)}{r+s-n}\right\rfloor$
				&\kern-.5cm $\left\lfloor \frac{\binom{v}{2}-e}{(n-r)\binom{\lfloor \frac{v}{n}\rfloor}{2}+r\binom{\lceil \frac{v}{n}\rceil}{2}}\right\rfloor$\\
\hline\hline 
\columnbox{equi-$n$ square} & $\mathbf{n+1}$ & $\mathbf{n+1}$ & --- & $\mathbf{n+1}$ 
\\\hline
\columnbox{row or column latin square of order $n$} & $\mathbf{n}$ & $\mathbf{n}$ & $\mathbf{n}$ & $\mathbf{n}$
\\\hline
\columnbox{latin square of order $n$} & $\mathbf{n-1}$ & $\mathbf{n-1}$ & $\mathbf{n-1}$ & $\mathbf{n-1}$
\\\hline
\columnbox{single diagonal latin square of order $n$} & $\mathbf{n-2}$ & $\mathbf{n-2}$ & $\mathbf{n-2}$ & $\mathbf{n-2}$
\\\hline
\columnbox{double diagonal latin square of odd order, $n>3$} & $\mathbf{n-3}$ & $n-2$ & $\mathbf{n-3}$ & $n-2$
\\\hline
\columnbox{sudoku square of order $n$} & $n-2$ & $n-2$ & $\mathbf{n-\sqrt{n}}$ & $n-2$
\\\hline
\columnbox{$m\times $n latin rectangle $m\le n$} & --- & $\mathbf{n-1}$ & $\mathbf{n-1}$ & $\mathbf{n-1}$
\end{tabular}
\end{table}

\section{Upper Bounds on sets of Orthogonal Latin Structures}\label{Section: Bounds for latin structures}
We first define several structures that are related to latin squares. Then we show that the four bounds give the best possible value of $N(G,n)$ for each structure $G$.

\subsection{Latin Squares}

An \emph{equi-$n$ square} is an $n\times n$ array such that each of $n$ symbols occurs in exactly $n$ cells. A \emph{row} (resp. \emph{column}) \emph{latin} square of order $n$ is an $n\times n$ array where each of $n$ symbols occurs in each row (resp. column). A \emph{single diagonal latin square} is a latin square where each symbol on the main diagonal is distinct. A \emph{double diagonal latin square} is a single diagonal latin square where each symbol along the back diagonal is also distinct.

Each of these latin structures is a balanced $n$-coloring of a graph with $n^2$ vertices. Two vertices of this underlying graph are adjacent if the structure forbids that they share a color. For instance, two vertices in the underlying graph for a single diagonal latin square are adjacent if they are in the same row or column or if they are both in the main diagonal.

\begin{corollary}\label{MOLS Corollary}
Let $n\ge2$ be a positive integer.
\begin{enumerate}[label=\textup{(\arabic*)}]
\item No set of mutually orthogonal equi-$n$ squares of order $n$ consists of more than $n+1$ squares. \label{equi-$n$ bound}
\item No set of mutually orthogonal row (or column) latin squares of order $n$ consists of more than $n$ squares. \label{row latin bound}
\item No set of mutually orthogonal latin squares of order $n$ consists of more than $n-1$ squares.\label{latin bound}
\item No set of mutually orthogonal single diagonal latin squares of order $n$ consists of more than $n-2$ squares.
\label{single diagonal bound}
\item No set of mutually orthogonal double diagonal latin squares of odd order $n>3$ consists of more than $n-3$ squares. (See \textup{\cite{Ger74}}) \label{double diagonal bound}
\end{enumerate}
\end{corollary}

\begin{proof}
To prove each part we shall use the degree bound with $m=0$. 

In an equi-$n$ square, no cells are adjacent, so $\Delta= 0$. For statement \ref{equi-$n$ bound} the degree bound is $\frac{n^2-(0)-1}{n-1}=n+1$.

In a row latin square the rows form an $n$-clique, and so each vertex has degree $\Delta=n-1$. Thus for \ref{row latin bound} the degree bound is $\frac{n^2-(n-1)-1}{n-1}=n$.

In a latin square each vertex is adjacent to all others in the same row and column, so each vertex has degree $\Delta=2n-2$. Thus for \ref{latin bound} the degree bound is $\frac{n^2-(2n-2)-1}{n-1}=n-1$.

In a single diagonal latin square, the upper left cell is adjacent to each vertex in the top row, the left column, or the main diagonal, so it has degree $\Delta=3n-3$. In statement \ref{single diagonal bound} we have a degree bound of $\frac{n^2-(3n-3)-1}{n-1}=n-2$.

In a double diagonal latin square of odd order the center cell is adjacent to $\Delta=4n-4$ cells, so we have a degree bound of $\frac{n^2-(4n-4)-1}{n-1}=n-3$ for \ref{double diagonal bound}.
\end{proof}

\subsection{Gerechte Designs and Sudoku Squares}

A \emph{gerechte design of order $n$}\label{D: gerechte design} is an $n\times n$ array that has been partitioned into $n$ regions of size $n$ and filled with $n$ symbols so that each row, column, and region contains each symbol exactly once. A \emph{sudoku square of order $n^2$}\label{D: sudoku square} is a gerechte design of order $n^2$ where the partitioned regions are the $n^2$ subarrays (starting from the top left corner) of size $n \times n$. 

As a corollary to  Theorem~\ref{Clique Bound Theorem} we get a bound on the size of a set of mutually orthogonal gerechte designs. The following bound was proved in \cite[Corollary 2.2]{BCC08}.

\begin{corollary}\label{Gerechte Corollary}
Given a partition of an $n\times n$ array into regions $S_1,\ldots,S_n$, each of size $n$, the size of a set of mutually orthogonal gerechte designs for this partition is at most $n-j$, where $j$ is the maximum size of the intersection of a row or column with one of the sets $S_1,\ldots,S_n$ with $j<n$.
\end{corollary}
\begin{proof}
Let $A$ be a row (or column). If $S_i$ intersects $A$ in $j$ cells $(j<n)$, then we let $B=A\setminus S_i$. Then with $r=|B|$, and $s=t=n$ the clique bound limits the number of mutually orthogonal gerechte designs of order $n$ with partition $S_1,\ldots,S_n$ to $\frac{|B|(n-j)}{|B|+n-n}=n-j$.
\end{proof}

\begin{corollary}\label{Sudoku Corollary}
The maximum size of a set of mutually orthogonal sudoku squares of order $n^2$ is $n^2-n$.
\end{corollary}
\begin{proof}
A sudoku square is a gerechte design. The upper left block intersects the top row in $n$ cells, so the clique bound on the cardinality of a set of mutually orthogonal $n^2$-colorings is precisely $n^2-n$.
\end{proof}

When $n$ is a power of a prime the upper bound of Corollary~\ref{Sudoku Corollary} is always obtainable (see \cite{BCC08} or \cite{PV09} for constructions). We also note that most of the results of Corollary~\ref{MOLS Corollary} follow as a corollary to Theorem~\ref{Clique Bound Theorem}.

\subsection{Latin Rectangles}

Let $m\le n$. An $m\times n$ \emph{latin rectangle}\label{D: latin rectangle} is defined to be an $m\times n$ array based on $n$ symbols such that no symbol occurs more than once in any row or column. Latin rectangles are a natural extension of the notion of latin squares, and they have been studied extensively \cite{DK74,LM98}.

Two $m\times n$ latin rectangles are said to be \emph{orthogonal} if, when superimposed, the $mn$ ordered pairs are distinct. An $m\times n$ latin square is precisely an $n$-coloring of an $m\times n$ rook's graph, and orthogonality of latin rectangles corresponds precisely with mutually orthogonal $n$-colorings of those graphs. 

It is customary to let $N(m,n)$\label{N(m,n)} be the maximum cardinality of a set of mutually orthogonal latin rectangles (\emph{MOLR}) of size $m\times n$. The following bound for mutually orthogonal latin rectangles is well-known \cite{LM98}.

\begin{corollary}\label{N(m,n) Bound} 
$N(m,n)\le n-1$.
\end{corollary}
\begin{proof}
In the underlying graph, $G$, for an $m\times n$ latin rectangle each vertex shares a row with $n-1$ vertices and column with $m-1$ vertices. Thus we may apply the average-degree bound with $D=m+n-2$ and $v=mn$. Then
\[
N(G,n)\le\left\lfloor\frac{(mn-(m+n-2)-1)n}{mn-n}\right\rfloor=n-1.\qedhere
\]
\end{proof}

For $r\le s\le t$ a \emph{latin rectangle of type $(r,s,t)$} is an $r\times s$ array filled with $t$ distinct symbols such that no symbol occurs more than once in any row or column. 
 
Latin rectangles of type $(r,s,t)$ occur quite naturally. Any $r\times s$ rectangle inside a latin square is a latin rectangle of type $(r,s,t)$ for some $t$. A latin rectangle of type $(n,n,n)$ is precisely a latin square, and a latin rectangle of type $(r,s,s)$ is an $r\times s$ latin rectangle. 

Let $N(r,s,t)$\label{N(r,s,t)} be the maximum size of a set of MOLR of type $(r,s,t)$.  Then $N(r,s)=N(r,s,s)$ because each $r\times s$ latin rectangle is based on $s$ symbols. The bound $N(1,s)=N(1,s,t)=\infty$ is trivial because any two $1\times s$ latin rectangles are orthogonal.  Therefore, we shall assume that $r>1$.

First we note some obvious bounds. Clearly $N(r,s,t)\ge N(r,s)$ because a rectangle based on $s$ symbols is also based on $t$ symbols because $s\le t$. In fact with $r$ and $s$ fixed, $N(r,s,t)$ is a nondecreasing function of $t$. With $s$ and $t$ fixed (and $r<s$) $N(r,s,t)$ is a nonincreasing function of $r$ because a set of mutually orthogonal $(r+1)\times s$ rectangles yields a set of mutually orthogonal $r\times s$ rectangles by deleting a row. Similarly, with $r$ and $t$ fixed (and $r\le s\le t$) $N(r,s,t)$ is a nondecreasing function of $s$ because a set of $r\times(s+1)$ MOLR yields a set of $r\times s$ MOLR by deleting a column.

The following result is a direct application of the clique bound of Theorem~\ref{Clique Bound Theorem}.

\begin{corollary}\label{(r,s,t) Corollary}
When $t<2s$ and $r>1$, 
\[
N(r,s,t)\le\frac{s^2-s}{2s-t}.
\]
\end{corollary}
\begin{proof}
The top two rows of an $r\times s$ rectangle are $s$-cliques, $A$ and $B$. Each vertex in $B$ shares a column with a vertex of $A$, so we may take $r=s$ and $j=1$.  The clique bound is $\frac{s^2-s}{2s-t}$.
\end{proof}

Note that we get the bound $N(r,s,s)\le s-1$ of Corollary~\ref{N(m,n) Bound} as a special case of Corollary~\ref{(r,s,t) Corollary}. Also noteworthy is the case when $t=s+1$ where the maximum possible size of a collection of mutually orthogonal latin rectangles of type $(r,s,s+1)$ is $s$. In particular, when $s$ is one less than a power of a prime the bound is tight because we can use a complete set of $s$ MOLS of order $s+1$ to obtain $s$ mutually orthogonal latin rectangles of type $(r,s,s+1)$ by deleting $s+1-r$ rows and $1$ column from each square.

\subsection{Orthogonal Partial Latin Squares}
A \emph{partial latin square of order $n$} is an $n\times n$ of (possibly empty) cells such that each of $n$ distinct symbols occurs at most once in each row and column. Thus a partial latin square is a partial $n$-coloring of an $n\times n$ rook's graph. Two partial latin squares are said to be orthogonal if whenever two cells share a symbol in one partial latin square, then in the other partial latin square the two cells do not share the same symbol (or at least one of the two cells is empty).

Mutually orthogonal latin polyominoes were studied in \cite{AG96} by Abdel-Ghaffar, motivated by an application to computer databases. In particular, Abdel-Ghaffar defined $M_n(p)$ to be the maximum number of mutually orthogonal partial latin squares of order $n$ with the same $p$ cells filled. Theorem~1 of \cite{AG96} gives the bound
\[
M_n(p)\le\left\lfloor\frac{p(p-1)}{\lfloor p/n\rfloor(2p-n-n\lfloor p/n\rfloor)}\right\rfloor-2.
\]
It can be showed that this bound for $M_n(p)$ agrees precisely with the edge bound for a partial latin square of order $n$ with $p$ filled cells.

\subsection{Lower Bounds for $N(G,n)$}\label{sec low bounds}
While upper bounds for $N(G,n)$ are nice, it would be ideal to have some lower bounds for $N(G,n)$. To find a lower bound for $N(G,n)$ we look to the work of Caro and Yuster. In \cite{CY99} Caro and Yuster studied the the $k$-orthogonal chromatic number $O\!\chi_k(G)$, which is the minimum number, $n$, such that $N(G,n)\ge k$. They obtained upper bounds for $O\!\chi_k(G)$ that depend on parameters of the graph such as the maximum degree of a vertex or the chromatic number $\chi(G)$. Perhaps the most useful result from \cite{CY99} is Theorem~1.1, part of which we display below.
\begin{theorem}[Part of Theorem 1.1 of \cite{CY99}] \label{CY}
Let $G$ be a graph with $v$ vertices, and with maximum degree $\Delta$. For $k\ge2$ the following upper bound holds:
\[
O\!\chi_k(G)\le \min\left\{2\sqrt{k-1}\max\{\Delta,\sqrt{v}\},(k-1)\left\lceil\frac{v}{\Delta+1}\right\rceil+\Delta\right\}.
\]
\end{theorem}

We can obtain a lower bound for $N(G,n)$ by reversing engineering the bound of Theorem~\ref{CY}.
\begin{corollary} \label{CY-corollary}
Let $G$ be a graph with $v$ vertices and maximum degree $\Delta$. If $N(G,n)\ge2$, then
\[
N(G,n)\ge\max\left\{\left(\frac{n}{2\max\{\Delta,\sqrt{v}\}}\right)^2+1,\frac{n-\Delta}{\lceil\frac{v}{\Delta+1}\rceil}+1\right\}.
\]
\end{corollary}

The lower bound of Corollary~\ref{CY-corollary} is of use only when the maximum degree $\Delta$ is smaller than $n$. Theorems~1.2--1.4 of \cite{CY99} give good bounds on $O\!\chi_k(G)$  when the chromatic number is small or if the graph, $G$, has a special structure like a $t$-partite graph or a $d$-degenerate graph.  Theorems~1.2--1.4 of \cite{CY99} can also be reverse engineered to obtain lower bounds for $N(G,n)$, though they do not apply as widely, nor are they easy to write.

\subsection{Lower Bounds for Orthogonal Chromatic Numbers}\label{Orthogonal Chromatic Number Section}
 In this section we note that we can solve for $n$ in each the four bounds for $N(G,n)$ that we have introduced, and thus find lower bounds for $O\!\chi_k(G)$.

The following corollary is obtained from the average-degree bound of Theorem~\ref{Average-Degree Bound Theorem}.
\begin{corollary} \label{avg-deg corollary}
Let $G$ be a graph with $v$ vertices with average degree $D$. Then 
\[
O\!\chi_k(G)\ge \left\lceil\frac{vk}{v-D-1+k}\right\rceil.
\]
\end{corollary}

The next corollary is obtained from the clique bound of Theorem~\ref{Clique Bound Theorem}.
\begin{corollary}\label{clique corollary}
Let $G$ be a graph with an $s$-clique, $A$, that is disjoint from an $r$-clique, $B$, such that each vertex in $B$ is adjacent to at least $j$ vertices in $A$. If $O\!\chi_k(G)<r+s$, then we get the bound
\[
O\!\chi_k(G)\ge \left\lceil r+s+\frac{rj-rs}{k}\right\rceil
\]
\end{corollary}

We can also derive lower bounds for $O\!\chi_k(G)$ from the degree bound of Theorem~\ref{Degree Bound Theorem} and the edge bound of Theorem~\ref{Edge Bound Theorem}, but the formulas are not as clean.

\begin{example}
Consider the cube graph, $G$, whose vertices and edges are the corners and edges of a cube. $G$ has no pair of orthogonal 3 colorings. The degree bound is $N(G,3)\le4$, while the clique bound, average-degree bound, and edge bound each yield the inequality $N(G,3)\le 2$, so none of these bounds are sharp in this case. With 4 colors the average-degree bound and the edge bound are $N(G,4)\le 4$. Moreover, we can find 4 mutually orthogonal 4-colorings of $G$.
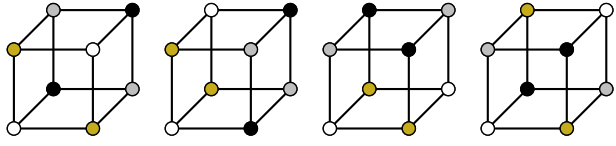
\[
\begin{tikzpicture}[scale=.35]
%\GraphInit[vstyle=Welsh]
\SetVertexSimple[MinSize = 5pt,
LineWidth = .5pt,FillColor = white]
\tikzset{VertexStyle/.append style =
{inner sep = 0pt,%
outer sep = 0pt}}
\SetVertexNoLabel
%============
\begin{scope}[scale=1.5]
\Vertex[x=0 ,y=0]{1}
\Vertex[x=2 ,y=0]{2}
\Vertex[x=2 ,y=2]{3}
\Vertex[x=0 ,y=2]{4}
\Vertex[x=1 ,y=1]{a}
\Vertex[x=3 ,y=1]{b}
\Vertex[x=3 ,y=3]{c}
\Vertex[x=1 ,y=3]{d}
\Edges(1,2,3,4,1,a,b,c,d,a)
\Edges(b,2)
\Edges(c,3)
\Edges(d,4)
\AddVertexColor{white}{1,3}
\AddVertexColor{Dark}{2,4}
\AddVertexColor{black}{a,c}
\AddVertexColor{Light}{b,d}
\end{scope}
%------------
\begin{scope}[scale=1.5,xshift=4cm]
\Vertex[x=0 ,y=0]{1}
\Vertex[x=2 ,y=0]{2}
\Vertex[x=2 ,y=2]{3}
\Vertex[x=0 ,y=2]{4}
\Vertex[x=1 ,y=1]{a}
\Vertex[x=3 ,y=1]{b}
\Vertex[x=3 ,y=3]{c}
\Vertex[x=1 ,y=3]{d}
\Edges(1,2,3,4,1,a,b,c,d,a)
\Edges(b,2)
\Edges(c,3)
\Edges(d,4)
\AddVertexColor{white}{1,d}
\AddVertexColor{Dark}{4,a}
\AddVertexColor{black}{2,c}
\AddVertexColor{Light}{3,b}
\end{scope}
%------------
\begin{scope}[scale=1.5,xshift=8cm]
\Vertex[x=0 ,y=0]{1}
\Vertex[x=2 ,y=0]{2}
\Vertex[x=2 ,y=2]{3}
\Vertex[x=0 ,y=2]{4}
\Vertex[x=1 ,y=1]{a}
\Vertex[x=3 ,y=1]{b}
\Vertex[x=3 ,y=3]{c}
\Vertex[x=1 ,y=3]{d}
\Edges(1,2,3,4,1,a,b,c,d,a)
\Edges(b,2)
\Edges(c,3)
\Edges(d,4)
\AddVertexColor{white}{1,b}
\AddVertexColor{Dark}{2,a}
\AddVertexColor{black}{3,d}
\AddVertexColor{Light}{4,c}
\end{scope}
%------------
\begin{scope}[scale=1.5,xshift=12cm]
\Vertex[x=0 ,y=0]{1}
\Vertex[x=2 ,y=0]{2}
\Vertex[x=2 ,y=2]{3}
\Vertex[x=0 ,y=2]{4}
\Vertex[x=1 ,y=1]{a}
\Vertex[x=3 ,y=1]{b}
\Vertex[x=3 ,y=3]{c}
\Vertex[x=1 ,y=3]{d}
\Edges(1,2,3,4,1,a,b,c,d,a)
\Edges(b,2)
\Edges(c,3)
\Edges(d,4)
\AddVertexColor{white}{1,c}
\AddVertexColor{Dark}{2,d}
\AddVertexColor{black}{3,a}
\AddVertexColor{Light}{4,b}
\end{scope}
\end{tikzpicture}
\]
Hence $N(G,4)=4$.

Corollary~\ref{avg-deg corollary} yields the following lower bounds on the orthogonal chromatic numbers $O\!\chi(G)\ge 3$ and $O\!\chi_4(G)\ge O\!\chi_3(G)\ge 4$.
The facts stated in the previous paragraph confirm that $O\!\chi(G)=O\!\chi_3(G)=O\!\chi_4(G)=4$.

It is worth noting that  we can obtain upper bounds on the $k$-orthogonal chromatic numbers of $G$ from \cite{CY99}, namely $O\!\chi(G)\le 4$ (by Theorem~1.3 of \cite{CY99}) $O\!\chi_3(G)\le 6$ (by Theorem~1.2 of \cite{CY99}) and $O\!\chi_4(G)\le 9$ (by Theorem~1.1 or Theorem~1.4 of \cite{CY99}).
\end{example}

\section{Proofs of the Main Theorems}\label{Section: Proofs}
\subsection{The Degree Bound}
\begin{proof}[Proof of Theorem~\ref{Degree Bound Theorem}]
Let $G$ be a graph of order $n^2-m$ where $0\le m< n-1$, and let $u$ be a vertex of degree $\Delta$. We must show
\[
N(G,n)\le \left\lfloor\frac{n^2-m-\Delta-1}{n-m-1}\right\rfloor.
\]

Let $C_1,\ldots,C_k$ be mutually orthogonal $n$-colorings of $G$. If $\Delta\ge n^2-n$ then by the generalized pigeon hole principle at least one color must be applied to more than $n$ vertices in $C_1$, so no coloring of $G$ may be orthogonal to $C_1$. Thus if $\Delta \ge n^2-n$, then $N(G,n)\le 1$.

Next assume that $k>1$, so that no color may be applied to more than $n$ vertices in any coloring $C_i$. Then, since no color may be applied to more than $n$ vertices, by the generalized pigeon hole principle there must be at least $n-m$ vertices that are assigned the color $C_i(u)$ in the coloring $C_i$. There are at most $n^2-m-\Delta-1$ vertices that can share a color with $u$ in any coloring $C_i$. The $n-m-1$ (or more) vertices of $G$ that share a color with vertex $u$ in $C_i$ cannot share a color with $u$ in $C_j$ whenever $i\ne j$, so we have the inequality
\[
k(n-m-1)\le n^2-m-\Delta-1.
\]
Solving for $k$ yields
\[
k\le \frac{n^2-m-\Delta-1}{n-m-1}.\qedhere
\]
\end{proof}

\subsection{The Clique Bound}
\begin{proof}[Proof of Theorem~\ref{Clique Bound Theorem}]
Let $r$, $s$, and $n$ be integers satisfying $1<r, s\le n<r+s$. Let $G$ be a graph with an $s$-clique, $A$, that is disjoint from an $r$-clique, $B$, such that each vertex in $B$ is adjacent to at least $j$ vertices in $A$. We must show that
\[
N(G,n)\le \left\lfloor\frac{r(s-j)}{r+s-n}\right\rfloor.
\]

Let $C_1,\ldots,C_k$ be mutually orthogonal $n$-colorings of $G$. We may assume without loss of generality that the vertices of $A$ are colored with the same colors, say $1,\ldots,s$, so that each pair of colors $(1,1),(2,2),\ldots,(s,s)$ occurs in $A$ for each pair of superimposed colorings $C_i,C_j$. Among the colorings $C_1,\ldots, C_k$, no vertex of $B$ can have any color in the set $\{1,\ldots,s\}$ more than once, so colors from the set $\{1,\ldots,s\}$ can be applied to a vertex of $B$ at most $s-j$ times across all the colorings $C_1,\ldots, C_k$. Thus, colors from the set $\{s+1,\ldots,n\}$ will occur at least $k-(s-j)$ times for each of the $r$ vertices of $B$ among the colorings $C_1,\ldots,C_k$. The total number of times that colors from the set $\{s+1,\ldots,n\}$ occur in $B$ among the colorings $C_1,\ldots,C_k$ is at least $r(k-s+j)$.

On the other hand, at most $n-s$ of the colors from the set $\{s+1,\ldots,n\}$ can occur in $B$ for each coloring $C_i$. Thus the total number occurences of a color from the set $\{s+1,\ldots,n\}$ in $B$ among the colorings $C_1,\ldots,C_k$ is at most $k(n-s)$. Therefore $r(k-s+j)\le m(n-s)$ so that
\[
k\le\frac{r(s-j)}{r+s-n}.\qedhere
\]
\end{proof}

\subsection{The Average-Degree Bound}
\begin{proof}[Proof of Theorem~\ref{Average-Degree Bound Theorem}]
Let $G$ be a graph of order $v$ with average vertex degree $D$.  We must show that for $n<v$,
\[
N(G,n)\le\left\lfloor\frac{(v-D-1)n}{v-n}\right\rfloor.
\] 

Let $C_1,\ldots,C_k$ be mutually orthogonal $n$-colorings of $G$. Note that if each vertex shares a color with $s$ vertices in each of the colorings $C_1,\ldots,C_k$ then $ks\le v-D-1$. Similarly if each vertex shares a color with $s$ vertices on average across the colorings $C_1,\ldots,C_k$, then $ks\le v-D-1$, so $k\le \frac{v-D-1}{s}$. We shall find a lower bound for $s$ and hence an upper bound for $k$.

 Fix $i$ and consider the coloring $C_i$. Let $l_j$ be the number of times that color $j$ is assigned to a vertex in $C_i$, so $l_1+\cdots+l_n=v$. Then $l_j$ vertices share a color with $l_j-1$ vertices, so on average across all vertices, each vertex shares a color in $C_i$ with $s$ vertices on average, where
\[
s=\frac{l_1(l_1-1)+l_2(l_2-1)+\cdots+l_n(l_n-1)}{nv}.
\]
Here $s$ is weighted average of the values $\frac{l_1-1}{n}$, $\frac{l_2-1}{n}$,\ldots, $\frac{l_n-1}{n}$, where the larger values are weighted more heavily, so
\[
s\ge\frac{l_1-1}{n}+\frac{l_2-1}{n}+\cdots+\frac{l_n-1}{n}=\frac{v}{n}-1.
\]
 Thus $\left(\frac{v}{n}-1\right)k+D \le v-1$ so
\[
k\le\frac{(v-D-1)n}{v-n}.\qedhere
\]
\end{proof}

\subsection{The Edge Bound}
\begin{proof}[Proof of Theorem~\ref{Edge Bound Theorem}]
Let $G$ be a graph with $v$ vertices and $e$ edges, $v>n$. We must show that if $r$ is the remainder when $v$ is divided by $n$, then 
\[
N(G,n)\le \left\lfloor \frac{\binom{v}{2}-e}{(n-r)\binom{\lfloor \frac{v}{n}\rfloor}{2}+r\binom{\lceil \frac{v}{n}\rceil}{2}}\right\rfloor.
\]

Let $C_1,\ldots,C_k$ be mutually orthogonal $n$-colorings of $G$. Each pair of vertices that share a color in $C_i$ must have distinct colors in $C_j$, $j\ne i$. Once each pair of vertices has shared a color in one of the $C_i$'s it is not possible to find an additional $n$-coloring orthogonal to each of the $C_i$'s.  Thus we need to identify the minimum number of pairs of vertices that will share a color in a given $C_i$.  If a color is applied to $m$ vertices, then that accounts for $\binom{m}{2}$ pairs. Therefore, if color $t$ is applied to $\mu_t$ vertices, then 
\begin{equation}\label{binomial sum}
\binom{\mu_1}{2}+\binom{\mu_2}{2}+\cdots+\binom{\mu_n}{2}
\end{equation}
pairs of vertices are accounted for in the coloring $C_i$. The condition $\mu_1+\cdots+\mu_n=v$ guarantees that the expression \eqref{binomial sum} is minimized when $n-r$ of the colors appear $\lfloor\frac{v}{n}\rfloor$ times and the other $r$ colors appear $\lceil\frac{v}{n}\rceil$ times. Therefore
\[
k\left[(n-r)\binom{\lfloor \frac{v}{n}\rfloor}{2}+r\binom{\lceil \frac{v}{n}\rceil}{2}\right]+e
\le \binom{v}{2}. 
\]
Solving for $k$ then yields the theorem.
\end{proof}

\section{Constructions with Orthogonal Colorings}\label{Section: Constructions} 
In this section we introduce a few graph structures that can be created from a set of orthogonal colorings.
\subsection{Coloring-Extended Supergraphs}\label{Coloring-Extended Supergraphs Section}
Given a graph $G$ and a collection of colorings $C_1,\ldots,C_k$ of $G$ we can define a graph $G_{C_1,\ldots,C_k}$,\label{GC1dotsCk} which we call the \emph{coloring-extended supergraph}\label{D: coloring-extended supergraph} to be the graph with vertex set $V(G)$, and distinct vertices $u$ and $v$ are adjacent in $G_{C_1,\ldots,C_k}$ if either $u\sim v$ in $G$ or $C_i(u)=C_i(v)$ for some $i$.
\begin{example}
Below we display two orthogonal colorings, $C_1$ and $C_2$, of a graph $G$ and the supergraph $G_{C_1,C_2}$.
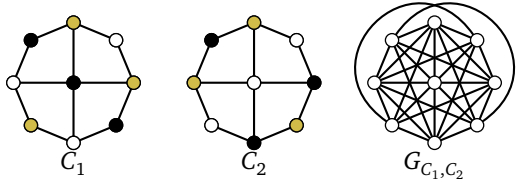
\[
%\includegraphics{UpperBoundsOnSetsOfOrthogonalColoringsOfGraphs-figure6.pdf}
%\begin{comment}
\begin{tikzpicture}[scale=.4]
\clip (-14,-3.5) rectangle (14,3);
%\GraphInit[vstyle=Welsh]
\SetVertexSimple[MinSize = 5pt,
LineWidth = .5pt]
\tikzset{VertexStyle/.append style =
{inner sep = 0pt,%
outer sep = 0pt}}
\SetVertexNoLabel
%============
\begin{scope}[xshift=-6cm]
\begin{scope}[rotate=45]
\Vertices[unit=2]{circle}{aa,bb,cc,ab,ba,ac,cb,bc}
\end{scope}
\coordinate (ca) at (0,0);
\Vertices[Node]{line}{ca}
\Edges(aa,bb,cc,ab,ba,ac,cb,bc,aa)
\Edges(bc,ca,ab)
\Edges(bb,ca,ac)
\AddVertexColor{white}{aa,ab,ac}
\AddVertexColor{Dark!80}{ba,bb,bc}
\AddVertexColor{black}{ca,cb,cc}
\draw (0,-2) node[below]{$C_1$};
\begin{scope}[xshift=6cm]
\begin{scope}[rotate=45]
\Vertices[unit=2]{circle}{aa,bb,cc,ab,ba,ac,cb,bc}
\end{scope}
\coordinate (ca) at (0,0);
\Vertices[Node]{line}{ca}
\Edges(aa,bb,cc,ab,ba,ac,cb,bc,aa)
\Edges(bc,ca,ab)
\Edges(bb,ca,ac)
\AddVertexColor{white}{aa,ba,ca}
\AddVertexColor{Dark!80!white}{ab,bb,cb}
\AddVertexColor{black}{ac,bc,cc}
\draw (0,-2) node[below]{$C_2$};
\end{scope}
\end{scope}
\begin{scope}[xshift=6cm]
\renewcommand{\VertexFillColor}{white}
\begin{scope}[rotate=45]
\Vertices[unit=2]{circle}{aa,bb,cc,ab,ba,ac,cb,bc}
\end{scope}
\coordinate (ca) at (0,0);
\Vertices[Node]{line}{ca}
\Edges(aa,bb,cc,ab,ba,ac,cb,bc,aa)
\Edges(bc,ca,ab)
\Edges(bb,ca,ac)
\Edges(aa,ab,ac,aa,ca,ba)
\Edges[style={bend left=100,looseness=2.2}](ba,aa)
\Edges(bb,bc,ba,bb,ab,cb,bb)
\Edges(cb,ca,cc,ac,bc,cc)
\Edges[style={bend left=100,looseness=2.2}](cc,cb)
\draw (0,-2) node[below]{$G_{C_1,C_2}$};
\end{scope}
\end{tikzpicture}
%\end{comment}
\]
\end{example}

\begin{theorem}[Caro and Yuster \cite{CY99}] \label{Supergraph Theorem}
Let $G$ be a graph with mutually orthogonal $n$-colorings $C_1,\ldots,C_k$. Then the collection, $C_1,\ldots,C_k$, of mutually orthogonal $n$-colorings of $G$ can be extended if and only if there exists an $n$-coloring of the graph $G_{C_1,\ldots,C_k}$.
\end{theorem}
\begin{proof}
Let $K$ be an $n$-coloring of $G_{C_1,\ldots,C_k}$. Then $K$ is also an $n$-coloring of $G$, and $(K(u),C_i(u))=(K(v),C_i(v))$ implies that $K(u)=K(v)$ and $C_i(u)=C_i(v)$. However, the condition $C_i(u)=C_i(v)$ implies that $u\sim v$ in $G_{C_1,\ldots,C_k}$, so the only possible way to satisfy $K(u)=K(v)$ is if $u=v$. Thus $K$ is orthogonal to each $C_i$.

For the converse we suppose that $K$ is an $n$-coloring of $G$ that is orthogonal to each of $C_1,\ldots,C_k$. If $K(u)=K(v)$ then $u\not\sim v$ in $G$, and, since $K$ is orthogonal to $C_i$, we must have $C_i(u)\ne C_i(v)$ for each $i$. Therefore $K$ is also an $n$-coloring of $G_{C_1,\ldots,C_k}$.
\end{proof}
\begin{example}
Below we display orthogonal colorings $C$ and $K$ of a graph $G$.
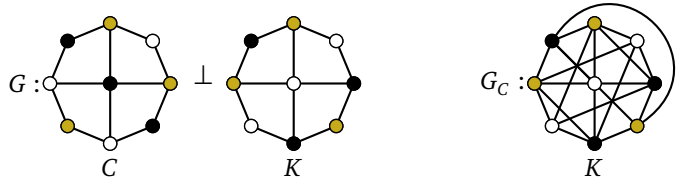
\[ 
%\includegraphics{UpperBoundsOnSetsOfOrthogonalColoringsOfGraphs-figure7.pdf}
%\begin{comment}
\begin{tikzpicture}[scale=.4]
%\GraphInit[vstyle=Welsh]
\SetVertexSimple[MinSize = 5pt,
LineWidth = .5pt]
\tikzset{VertexStyle/.append style =
{inner sep = 0pt,%
outer sep = 0pt}}
\SetVertexNoLabel
%============
\begin{scope}[xshift=-6.1cm]
\draw (3.1,.3) node {$\perp$};
\begin{scope}[rotate=45]
\Vertices[unit=2]{circle}{aa,bb,cc,ab,ba,ac,cb,bc}
\end{scope}
\coordinate (ca) at (0,0);
\Vertices[Node]{line}{ca}
\Edges(aa,bb,cc,ab,ba,ac,cb,bc,aa)
\Edges(bc,ca,ab)
\Edges(bb,ca,ac)
\AddVertexColor{white}{aa,ab,ac}
\AddVertexColor{Dark}{ba,bb,bc}
\AddVertexColor{black}{ca,cb,cc}
\draw (0,-2.2) node[below]{$C$};
\draw (-2,0)node[left]{$G:$};
\end{scope}
\begin{scope}[xshift=0cm]
\begin{scope}[rotate=45]
\Vertices[unit=2]{circle}{aa,bb,cc,ab,ba,ac,cb,bc}
\end{scope}
\coordinate (ca) at (0,0);
\Vertices[Node]{line}{ca}
\Edges(aa,bb,cc,ab,ba,ac,cb,bc,aa)
\Edges(bc,ca,ab)
\Edges(bb,ca,ac)
\AddVertexColor{white}{aa,ba,ca}
\AddVertexColor{Dark}{ab,bb,cb}
\AddVertexColor{black}{ac,bc,cc}
\draw (0,-2.2) node[below]{$K$};
\end{scope}
\begin{scope}[xshift=10cm]
\begin{scope}[rotate=45]
\Vertices[unit=2]{circle}{aa,bb,cc,ab,ba,ac,cb,bc}
\end{scope}
\coordinate (ca) at (0,0);
\Vertices[Node]{line}{ca}
\Edges(aa,bb,cc,ab,ba,ac,cb,bc,aa)
\Edges(bc,ca,ab)
\Edges(bb,ca,ac)
\Edges(aa,ab,ac,aa)
\Edges(bb,bc,ba,bb)
\Edges(cb,cc,ca)
\Edges[style={bend left=100,looseness=2.2}](cc,cb)
\AddVertexColor{white}{aa,ba,ca}
\AddVertexColor{Dark}{ab,bb,cb}
\AddVertexColor{black}{ac,bc,cc}
\draw (0,-2.2) node[below]{$K$};
\draw (-2,0)node[left]{$G_C:$};
\end{scope}
\end{tikzpicture}
%\end{comment}
\]
Note that $K$ is also a coloring of the graph $G_C$.
\end{example}

In \cite{CY99} Caro and Yuster made use of Theorem~\ref{Supergraph Theorem} to find an upper bound for the orthogonal chromatic number $O\!\chi(G)$.

\begin{remark}
When $n^2\ge|G|=v\ge\frac{n^2}{2}+2$ the clique bound can be applied to the graph $G_{C_1}$. In this case we obtain the bound
\[
N(G,n)\le 
\left\lfloor
\frac{
\left\lceil \frac{v}{n}\right\rceil
\left\lceil \frac{v-\left\lceil \frac{v}{n}\right\rceil}{n-1}\right\rceil
}
{
\left\lceil \frac{v}{n}\right\rceil +
\left\lceil \frac{v-\left\lceil \frac{v}{n}\right\rceil}{n-1}\right\rceil-n
}
\right\rfloor
+1.
\]
While this is a good bound, it does not outperform the edge bound, though it does match the edge bound for large values of $v$ when the initial graph $G$ has few edges. In particular, when $v=n^2$ we get $N(G,n)\le n+1$, so part~\ref{equi-$n$ bound} of Corollary~\ref{MOLS Corollary} can also be viewed as a corollary of the clique bound.
\end{remark}

\subsection{MOLS via Mutually Orthogonal Colorings}
A set of mutually orthogonal colorings of a graph can be used to create a set of MOLS under the right conditions.  The existence of $k-2$ mutually orthogonal latin squares of order $n$ is equivalent to the existence of $k$ mutually orthogonal equi-$n$ squares of order $n$. A set of $k-2$ mutually orthogonal latin squares of order $n$ can be extended to a set of $k$ mutually orthogonal equi-$n$ squares by adding a square with constant rows and a square with constant columns. It is similarly possible to construct $k-2$ mutually orthogonal latin squares from a set of $k$ mutually orthogonal equi-$n$ squares. The following theorem and its proof provide a generalization of this construction using mutually orthogonal colorings of graphs.

\begin{theorem}\label{MOLP to MOLS Theorem}
Let $G$ be a graph of order $n^2$. Suppose there exists a set of $k$ mutually orthogonal $n$-colorings of $G$. Then there exists a set of $k$ mutually orthogonal equi-$n$ squares, $k-1$ mutually orthogonal row (or column) latin squares, $k-2$ MOLS of order $n$, and $k-3$ single diagonal latin squares.
\end{theorem}
\begin{proof}
We first construct a set of $k-2$ MOLS. Let $C_1,\ldots,C_k$ be mutually orthogonal $n$-colorings of $G$. We can obtain a set of $k-2$ MOLS $L_1,\ldots,L_{k-2}$ of order $n$ by letting the entry in row $i$ and column $j$ of $L_m$ be the color $C_m(v)$, where $v$ is the unique vertex in $G$ with color $C_{k-1}(v)=i$ and $C_k(v)=j$. To see that the squares $L_1,\ldots,L_{k-2}$ are latin note that each row and column will have distinct entries because $C_{k-1}$ and $C_k$ are orthogonal colorings. Moreover, since $C_1,\ldots,C_{k-2}$ are pairwise orthogonal, the squares $L_1,\ldots,L_{k-2}$ will also be pairwise orthogonal.

To this set of MOLS we can add an equi-$n$ square with constant columns (or constant rows) to obtain a set of $k-1$ row (or column) latin squares. Moreover, if we include both the square with constant rows and the square with constant columns we have a set of $k$ mutually orthogonal equi-$n$ squares.   

To obtain a set of $k-3$ mutually orthogonal single diagonal latin squares we can identify a permutation, $\sigma$, that will permute the rows of $L_{k-2}$ so that the main diagonal entries of $\sigma(L_{k-2})$ are equal. Then apply $\sigma$ to each square $L_1,\ldots,L_{k-3}$. The resulting squares $\sigma(L_1),\ldots,\sigma(L_{k-2})$ will still be mutually orthogonal, and each of $\sigma(L_1),\ldots,\sigma(L_{k-3})$ will be a single diagonal latin square because no symbol will occur more than once down the main diagonal. 
\end{proof}

\subsection{Or-Products}\label{Or-Product Section}
There are several ways to define the product of two graphs. The most commonly studied products are the tensor product graph, the cartesian product graph, the strong product graph, and the lexicographical product graph. We refer the reader to the reference \cite{IK00} for definitions and results about these product graphs. Here we define yet another product graph.

Let $G$ and $H$ be graphs. We define the \emph{or-product}\label{D: or-product} of $G$ and $H$ to be the graph $G\circledcirc H$\label{G circle H} that has vertex set $V(G)\times V(H)$ where vertices $(u,v)$ and $(u',v')$ are adjacent in $G\circledcirc H$ if either $u\sim u'$ in $G$ or $v\sim v'$ in $H$. The or-product is commutative and associative and each of the product graphs named in the paragraph above is a subgraph of the or-product graph. (In fact $G\circledcirc H$ is the complement of the strong product of the complements of $G$ and $H$.)

\begin{example}\label{or-product example}
The Kronecker product of matrices can be used to create a set of MOLS of order $mn$ from a set of MOLS of order $m$ together with a set of MOLS of order $n$. However, the $mn\times mn$ rook's graph is not the same as the graph that one obtains by taking the graph theoretic cartesian product of an $m\times m$ rook's graph and an $n\times n$ rook's graph (nor is it the same graph as any other product of rooks graphs).

Below we show a $4\times 4$ rook's graph, $G$, with a vertex,
%
%\includegraphics{UpperBoundsOnSetsOfOrthogonalColoringsOfGraphs-figure8.pdf},
%\begin{comment}
\begin{tikzpicture}[baseline=(current bounding box.south)]
\begin{scope}[scale=.3]
\draw (0,0) rectangle (1,1);
\clip (0,0) rectangle (1,1);
\draw[ystep=.25,xstep=4,rotate=45] (-1,-1) grid (2,2);
\end{scope}
\end{tikzpicture},
%\end{comment} 
%
that is adjacent to each of the vertices of the form
%
%\includegraphics{UpperBoundsOnSetsOfOrthogonalColoringsOfGraphs-figure9.pdf}
%\begin{comment}
\begin{tikzpicture}[scale=.3,baseline=(current bounding box.south)]
\draw (0,0) rectangle (1,1);
\draw (.5,.5) node[inner sep=0pt]{$\times$}; 
\end{tikzpicture} 
%\end{comment}
%
and a $3\times 3$ rooks graph, $H$,  with a vertex,
%
%\includegraphics{UpperBoundsOnSetsOfOrthogonalColoringsOfGraphs-figure10.pdf},
%\begin{comment}
\begin{tikzpicture}[rotate=90,baseline=(current bounding box.south)]
\begin{scope}[scale=.3]
\draw (0,0) rectangle (1,1);
\clip (0,0) rectangle (1,1);
\draw[ystep=.25,xstep=4,rotate=45] (-1,-1) grid (2,2);
\end{scope}
\end{tikzpicture},
%\end{comment}
% 
that is adjacent to each of the vertices of the form
%
%\includegraphics{UpperBoundsOnSetsOfOrthogonalColoringsOfGraphs-figure11.pdf}.
%\begin{comment}
\begin{tikzpicture}[scale=.3,baseline=(current bounding box.south)]
\draw (0,0) rectangle (1,1);
\draw (.5,.5) node[inner sep=0pt]{+};
\end{tikzpicture}.
%\end{comment}
%
We have marked the vertices of the graph $G\circledcirc H$ that are adjacent to the vertex
$
%\includegraphics{UpperBoundsOnSetsOfOrthogonalColoringsOfGraphs-figure12.pdf}
%\begin{comment}
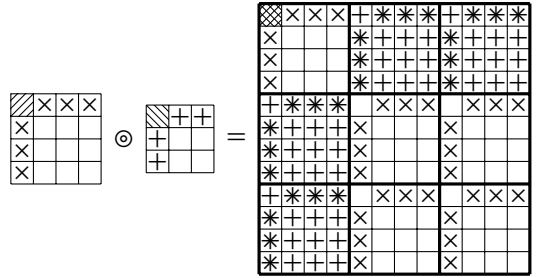
\begin{tikzpicture}[scale=.3,baseline=(current bounding box.south)]
\begin{scope}[yshift=11cm]
\draw (0,0) rectangle (1,1);
\clip (0,0) rectangle (1,1);
\draw[ystep=.25,xstep=4,rotate=45] (-1,-1) grid (2,2);
\end{scope}
\begin{scope}[yshift=11cm,xshift=1cm,rotate=90]
\clip (0,0) rectangle (1,1);
\draw[ystep=.25,xstep=4,rotate=45] (-1,-1) grid (2,2);
\end{scope}
\end{tikzpicture}
%\end{comment}
=
\raisebox{.05cm}{\big(}
%\includegraphics{UpperBoundsOnSetsOfOrthogonalColoringsOfGraphs-figure13.pdf},
%\begin{comment}
\begin{tikzpicture}[baseline=(current bounding box.south)]
\begin{scope}[scale=.3]
\draw (0,0) rectangle (1,1);
\clip (0,0) rectangle (1,1);
\draw[ystep=.25,xstep=4,rotate=45] (-1,-1) grid (2,2);
\end{scope}
\end{tikzpicture}, 
%\end{comment}
%\includegraphics{UpperBoundsOnSetsOfOrthogonalColoringsOfGraphs-figure14.pdf}.
%\begin{comment}
\begin{tikzpicture}[rotate=90,baseline=(current bounding box.south)]
\begin{scope}[scale=.3]
\draw (0,0) rectangle (1,1);
\clip (0,0) rectangle (1,1);
\draw[ystep=.25,xstep=4,rotate=45] (-1,-1) grid (2,2);
\end{scope}
\end{tikzpicture}
%\end{comment}
\raisebox{.05cm}{\big)}$.

\[
%\includegraphics{UpperBoundsOnSetsOfOrthogonalColoringsOfGraphs-figure15.pdf}.
%\begin{comment}
\begin{tikzpicture}[scale=.3]
\draw (0,0) grid (12,12);
\draw[step=4,style=very thick] (0,0) grid (12,12);
\foreach \y in {4,8,12}
\foreach \x in {2,3,4,6,7,8,10,11,12}
{
\draw (\y-3,\x-1)++(-.5,-.5) node{$\times$};
\draw (\x,\y)++(-.5,-.5) node{$\times$};
}
\foreach \y in {9,10,11,12}
\foreach \x in {5,6,...,12}
{
\draw (\y-8,\x-4)++(-.5,-.5) node{$+$};
\draw (\x,\y)++(-.5,-.5) node{$+$};
}
%Corner
\begin{scope}[yshift=11cm]
\clip (0,0) rectangle (1,1);
\draw[ystep=.25,xstep=4,rotate=45] (-1,-1) grid (2,2);
\end{scope}
\begin{scope}[yshift=11cm,xshift=1cm,rotate=90]
\clip (0,0) rectangle (1,1);
\draw[ystep=.25,xstep=4,rotate=45] (-1,-1) grid (2,2);
\end{scope}
\begin{scope}[xshift=-11cm,yshift=4cm]
\draw (0,0) grid (4,4);
\foreach \x in {2,3,4}
{
\draw (\x,4)++(-.5,-.5) node{$\times$};
\draw (1,\x-1)++(-.5,-.5) node{$\times$};
}
%Corner
\begin{scope}[yshift=3cm]
\clip (0,0) rectangle (1,1);
\draw[ystep=.25,xstep=4,rotate=45] (-1,-1) grid (2,2);
\end{scope}
\end{scope}
\begin{scope}[xshift=-5cm,yshift=4.5cm]
\draw (0,0) grid (3,3);
\foreach \x in {2,3}
{
\draw (\x,3)++(-.5,-.5) node{$+$};
\draw (1,\x-1)++(-.5,-.5) node{$+$};
}
%Corner
\begin{scope}[yshift=2cm,xshift=1cm,rotate=90]
\clip (0,0) rectangle (1,1);
\draw[ystep=.25,xstep=4,rotate=45] (-1,-1) grid (2,2);
\end{scope}
\end{scope}
\draw (-6,6) node {$\circledcirc$};
\draw (-1,6) node {$=$};
\end{tikzpicture}
.
%\end{comment}
\]
\end{example}

The following result for graphs is a strengthening of the Kronecker product construction of mutually orthogonal latin squares.

\begin{theorem}\label{Or Product Theorem}
Let $G$ and $H$ be graphs and $m$ and $n$ be positive integers. If $S$ is a subgraph of $G\circledcirc H$, then
\[
N(S,mn)\ge \min\{N(G,m),N(H,n)\}.
\]
\end{theorem}
\begin{proof}
It is sufficient to show that 
\[
N(G\circledcirc H,mn)\ge\min \{N(G,m),N(H,n)\}.
\]
We create $mn$ colors by taking ordered pairs of colors where the first color is a color applied to vertices of $G$ and the second color is a color applied to vertices of $H$. We may then use a coloring, $A$, of $G$ and a coloring, $B$, of $H$ to obtain a coloring $AB$ of $G\circledcirc H$ by assigning to the vertex $(u,v)$ the color $(A(u),B(v))$. To see that $AB$ is a proper coloring of $G\circledcirc H$ we note that if $(u,v)\ne (u',v')$ and $AB(u,v)=AB(u',v')$ then $A(u)=A(u')$ and $B(v)=B(v')$, so $u\nsim u'$ and $v\nsim v'$ which implies $(u,v)\nsim(u',v')$.

Now assume that we have mutually orthogonal colorings $A_1,\ldots,A_r$ of $G$ and mutually orthogonal colorings $B_1,\ldots,B_r$ of $H$. We claim that the product colorings $A_1B_1,\ldots, A_rB_r$ are mutually orthogonal.

Suppose that for $i\ne j$ we have equal pairs of colors 
\[
\big(A_iB_i(u,v),A_jB_j(u,v)\big)=\big(A_iB_i(u',v'),A_jB_j(u',v')\big).
\]
To conclude that $A_iB_i$ is orthogonal to $A_jB_j$, we must show that $(u,v)=(u',v')$. 

Since $A_iB_i(u,v)=A_iB_i(u',v')$ then $A_i(u)=A_i(u')$ and $B_i(v)=B_i(v')$ while $A_jB_j(u,v)=A_jB_j(u',v')$ implies $A_j(u)=A_j(u')$ and $B_j(v)=B_j(v')$. Since $A_i$ is orthogonal to $A_j$, the equality of the ordered pairs $(A_i(u),A_j(u))=(A_i(u'),A_j(u')))$ implies $u=u'$. Simililarly $(B_i(v),B_j(v))=(B_i(v'),B_j(v'))$ implies $v=v'$. Therefore the coloring $A_iB_i$ is orthogonal to the coloring $A_jB_j$.
\end{proof}

As a special case of Theorem~\ref{Or Product Theorem} we obtain the well-known lower bound on the size of a set of MOLS that can be obtained via a Kronecker product construction \cite{LM98}.

\begin{corollary}\label{Kronecker Corollary}
$N(mn)\ge\min\{N(m),N(n)\}$.
\end{corollary}
\begin{proof}
Let $R_j$ denote the $j\times j$ rook's graph for each $j$. The graph, $R_{mn}$, is a subgraph of $R_m\circledcirc R_n$ in a natural way, so by Theorem~\ref{Or Product Theorem} 
\[
N(R_{mn},mn)\ge \min\{N(R_m,m),N(R_n,n)\}.\qedhere
\]  
\end{proof}

\subsection*{Acknowledgements}
The author is indebted to Yair Caro for pointing out several relevant references.

\bibliographystyle{plain}  
\bibliography{Biblio}

\end{document}